\documentclass[11pt]{article}
\usepackage{amsmath}
\usepackage{amssymb}
\usepackage{amsthm}
\usepackage{bookmark}
\usepackage{geometry}
\usepackage{hyperref}
\newtheorem{theorem}{Theorem}
\newtheorem{lemma}[theorem]{Lemma}
\newtheorem{proposition}[theorem]{Proposition}
\DeclareMathOperator{\N}{N}
\DeclareMathOperator{\Tr}{Tr}

\title{Permutation Rational Functions over Quadratic Extensions of Finite Fields}
\author{Ruikai Chen\textsuperscript{1, 2}\and Sihem Mesnager\textsuperscript{1, 2, 3}}
\date{\small\textsuperscript{1}Department of Mathematics, University of Paris VIII, F-93526 Saint-Denis\\\textsuperscript{2}Laboratory Analysis, Geometry and Applications, LAGA, University Sorbonne Paris Nord, CNRS, UMR 7539, F-93430, Villetaneuse, France\\\textsuperscript{3}Telecom Paris, Polytechnic institute of Paris, 91120 Palaiseau, France\\Emails: \href{mailto:chen.rk@outlook.com}{chen.rk@outlook.com}\quad\href{mailto:smesnager@univ-paris8.fr}{smesnager@univ-paris8.fr}}

\begin{document}

\maketitle

\noindent\textbf{Abstract.} Permutation rational functions over finite fields have attracted much attention in recent years. In this paper, we introduce a class of permutation rational functions over $\mathbb F_{q^2}$, whose numerators are so-called $q$-quadratic polynomials. To this end, we will first determine the exact number of zeros of a special $q$-quadratic polynomial in $\mathbb F_{q^2}$, by calculating character sums related to quadratic forms of $\mathbb F_{q^2}/\mathbb F_q$. Then given some rational function, we can demonstrate whether it induces a permutation of $\mathbb F_{q^2}$.

\section{Introduction}

A permutation polynomial over a finite field $\mathbb F_q$ is a polynomial that acts as a permutation on $\mathbb F_q$ in the usual way (see \cite[Chapter 7]{lidl1997}). Permutation polynomials play an important role in the theory of finite fields as well as applications in coding theory, cryptography, combinatorics, etc. For a survey of permutation polynomials, one may refer to \cite{hou2015permutation}.

There is a generalization of polynomials over finite fields, which can also induce permutations. A rational function over $\mathbb F_q$ is defined as a fraction whose numerator and denominator are polynomials over $\mathbb F_q$. In other words, a rational function $f$ over $\mathbb F_q$ can be written uniquely in the form $a\frac gh$, where $g$ and $h$ are relatively prime monic polynomials over $\mathbb F_q$ with $a\in\mathbb F_q^*$. It defines a function on $\mathbb F_q\cup\{\infty\}$ as
\[f(c)=\begin{cases}a\frac{g(c)}{h(c)}&\text{if }h(c)\ne0,\\\infty&\text{if }h(c)=0,\end{cases}\]
for $c\in\mathbb F_q$ and
\[f(\infty)=\begin{cases}\infty&\text{if }\deg(f)>\deg(g),\\a&\text{if }\deg(f)=\deg(g),\\0&\text{if }\deg(f)<\deg(g).\end{cases}\]
If this induces a permutation of $\mathbb F_q\cup\{\infty\}$, then $f$ is called a permutation rational function of $\mathbb F_q$. A few constructions of permutation rational functions are discovered recently. For example, permutation rational functions of degree $3$ and $4$ have been classified completely in \cite{ferraguti2020full} and \cite{hou2021rational} respectively. Also, it has been studied in \cite{hou2020type} along with \cite{bartoli2021conjecture} whether $f(x)=x+(x^p-x+b)^{-1}$, with $b\in\mathbb F_q$ and $p$ the characteristic of $\mathbb F_q$, is a permutation rational function over $\mathbb F_q$. In \cite{golouglu2022classification} there is a classification of permutation rational functions whose numerators and denominators are both in the form
\[a_{q+1}x^{q+1}+a_qx^q+a_1x+a_0.\]

Permutation rational functions, attracting much less attention, however, have almost the same applications as permutation polynomials. In particular, if a permutation rational function of $\mathbb F_q$ maps $\infty$ to $\infty$, then it is also a permutation of $\mathbb F_q$. In this paper, we are interested in such a class of rational functions over $\mathbb F_{q^2}$, potentially acting as permutation polynomials. Let $g$ be a $q$-quadratic polynomials over $\mathbb F_{q^2}$ (i.e., $g(x)$ is congruent to a quadratic polynomial modulo $x^q-x$) and consider the rational function defined by
\[f(x)=\frac{g(x)}{x^q+x+d}\]
with $\deg(g)>q$ and $d\in\mathbb F_{q^2}\setminus\mathbb F_q$. Since $f(\infty)=\infty$ and $x^q+x+d$ has no zero in $\mathbb F_{q^2}$, it is easy to see that $f$ as a function from $\mathbb F_{q^2}$ to itself is well-defined. This means for $t\in\mathbb F_{q^2}$, the zeros of $f(x)+t$ in $\mathbb F_{q^2}$ are exactly the same as those of $g(x)+t(x^q+x+d)$. It follows that $f$ permutes $\mathbb F_{q^2}$ if and only if $g(x)+t(x^q+x+d)$ has exactly one zero in $\mathbb F_{q^2}$ for every $t\in\mathbb F_{q^2}$. Therefore, in Section 2, we will first attempt to determine the exact number of zeros of $q$-quadratic polynomials over $\mathbb F_{q^2}$, generally in the form
\[x^{q+1}+\alpha x^{2q}+\beta x^2+ax^q+bx+c\]
or
\[\alpha x^{2q}+\beta x^2+ax^q+bx+c.\]
Then in Section 3, for some specific rational functions over $\mathbb F_{q^2}$, we are able to tell whether they induce permutations of $\mathbb F_{q^2}$.

The following notation will be used throughout. For the polynomial ring $\mathbb F_{q^2}[x]$, let $x^{q^2}-x=0$ by abuse of notation, since all these polynomials are treated as maps from $\mathbb F_{q^2}$ to itself. Let $\Tr$ and $\N$ be defined by
\[\Tr(x)=x^q+x\quad\text{and}\quad\N(x)=x^{q+1}.\]
Let $\chi$ and $\psi$ be the canonical additive characters of $\mathbb F_{q^2}$ and $\mathbb F_q$ respectively, which means $\chi=\psi\circ\Tr$. In the case that $q$ is odd, let $\eta$ be the quadratic character of $\mathbb F_q^*$, and $G$ be the quadratic Gauss sum of $\mathbb F_q$ defined as
\[G=\sum_{u\in\mathbb F_q^*}\eta(u)\psi(u),\]
with $G^2=\eta(-1)q$.


\section{On the Number of Zeros of $q$-quadratic Polynomials in $\mathbb F_{q^2}$}

Assume $q$ is odd in this section. Let $g$ be a $q$-quadratic polynomials over $\mathbb F_{q^2}$. Let $L(x)=c_1x^q+c_0x$ be nonzero with $c_0,c_1\in\mathbb F_{q^2}$. Clearly $g\circ L$ and $L\circ g$ are both $q$-quadratic polynomials. If $\N(c_0)\ne\N(c_1)$, then $L$ is a linear endomorphism on $\mathbb F_{q^2}/\mathbb F_q$ that permutes $\mathbb F_{q^2}$, and the number of zeros of $g$ in $\mathbb F_{q^2}$ is exactly the same as that of $g\circ L$ or $L\circ g$. Thus, some specific forms of polynomials,
\[x^{q+1}+x^2+ax^q+bx+c\]
and
\[\alpha x^{2q}+\beta x^2+ax^q+bx+c\]
with $\N(\alpha)=\N(\beta)$, will be investigated, and others may be dealt with via some linear maps, as shown later. Before that, an essential lemma will be introduced.

\begin{lemma}\label{sum}
Let $A,B,C\in\mathbb F_{q^2}$ with $B\ne0$ and $D=\Tr(A)^2-4\N(B)$. If there does not exist $\theta\in\mathbb F_{q^2}$ such that $\Tr(A)\theta^q+2B\theta+C=0$, then $D=0$ and
\[\sum_{w\in\mathbb F_{q^2}}\chi(Aw^{q+1}+Bw^2+Cw)=0.\]
If such $\theta$ exists, then
\[\sum_{w\in\mathbb F_{q^2}}\chi(Aw^{q+1}+Bw^2+Cw)=\begin{cases}qG\chi(-A\theta^{q+1}-B\theta^2)&\text{if }\Tr(A)=2B^{\frac{q+1}2},\\-qG\chi(-A\theta^{q+1}-B\theta^2)&\text{if }\Tr(A)=-2B^{\frac{q+1}2},\\-\eta(D)q\chi(-A\theta^{q+1}-B\theta^2)&\text{if }\Tr(A)^2\ne4\N(B),\end{cases}\]
where
\[\theta=\frac{-\Tr(A)C^q+2B^qC}{\Tr(A)^2-4\N(B)}\]
in the case $D\ne0$.
\end{lemma}
\begin{proof}
Note that
\[\begin{split}&\mathrel{\phantom{=}}\Tr\left(A(x+\theta)^{q+1}+B(x+\theta)^2+C(x+\theta)\right)\\&=\Tr\left(Ax^{q+1}+Bx^2+(\Tr(A)\theta^q+2B\theta+C)x+A\theta^{q+1}+B\theta^2+C\theta\right)\end{split}\]
for $\theta\in\mathbb F_{q^2}$. If $\Tr(A)\theta^q+2B\theta+C=0$ for some $\theta\in\mathbb F_{q^2}$, then $B\theta^2+C\theta=-\Tr(A)\theta^{q+1}-B\theta^2$ and
\[\begin{split}&\mathrel{\phantom{=}}\sum_{w\in\mathbb F_{q^2}}\chi(Aw^{q+1}+Bw^2+Cw)\\&=\sum_{w\in\mathbb F_{q^2}}\chi(A(w+\theta)^{q+1}+B(w+\theta)^2+C(w+\theta))\\&=\sum_{w\in\mathbb F_{q^2}}\chi(Aw^{q+1}+Bw^2+A\theta^{q+1}+B\theta^2+C\theta)\\&=\sum_{w\in\mathbb F_{q^2}}\chi(Aw^{q+1}+Bw^2-A^q\theta^{q+1}-B\theta^2)\\&=\sum_{w\in\mathbb F_{q^2}}\chi(Aw^{q+1}+Bw^2-A\theta^{q+1}-B\theta^2),\end{split}\]
where
\[\sum_{w\in\mathbb F_{q^2}}\chi(Aw^{q+1}+Bw^2)=\begin{cases}qG&\text{if }\Tr(A)=2B^{\frac{q+1}2},\\-qG&\text{if }\Tr(A)=-2B^{\frac{q+1}2},\\-\eta(D)q&\text{if }\Tr(A)^2\ne4\N(B)\ne0,\end{cases}\]
according to \cite[Theorem 1 and Proposition 6]{chen2023evaluation}.

In the case $\Tr(A)^2\ne4\N(B)$, the polynomial $\Tr(A)x^q+2Bx+C$ is a permutation polynomial of $\mathbb F_{q^2}$, and its unique zero is
\[\theta=\frac{-\Tr(A)C^q+2B^qC}{\Tr(A)^2-4\N(B)}\]
as easily seen. Then
\[\begin{split}&\mathrel{\phantom{=}}\sum_{w\in\mathbb F_{q^2}}\chi(Aw^{q+1}+Bw^2+Cw)\\&=\sum_{w\in\mathbb F_{q^2}}\chi(Aw^{q+1}+Bw^2)\chi(-A\theta^{q+1}-B\theta^2)\\&=-\eta(D)q\chi(-A\theta^{q+1}-B\theta^2).\end{split}\]

Suppose now that $\Tr(A)^2=4\N(B)$. With $A$ and $B$ fixed, let
\[S=\{C\in\mathbb F_{q^2}|\Tr(A)\theta^q+2B\theta+C=0\text{ for some }\theta\in\mathbb F_{q^2}\}.\]
The image of $\Tr(A)x^q+2Bx$ as a linear endomorphism of $\mathbb F_{q^2}/\mathbb F_q$ has exactly $q$ elements, so $|S|=q$. For $C\in S$, we have
\[\Big|\sum_{w\in\mathbb F_{q^2}}\chi(Aw^{q+1}+Bw^2+Cw)\Big|^2=q^3,\]
and thus,
\[\sum_{C\in S}\left|\sum_{w\in\mathbb F_{q^2}}\chi(Aw^{q+1}+Bw^2+Cw)\right|^2=q^4,\]
while
\[\sum_{C\in\mathbb F_{q^2}}\left|\sum_{w\in\mathbb F_{q^2}}\chi(Aw^{q+1}+Bw^2+Cw)\right|^2=q^2\sum_{w\in\mathbb F_{q^2}}\left|\chi(Aw^{q+1}+Bw^2)\right|^2=q^4,\]
by Parseval's identity. Hence, if $C\notin S$, then
\[\sum_{w\in\mathbb F_{q^2}}\chi(Aw^{q+1}+Bw^2+Cw)=0,\]
as desired.
\end{proof}

\subsection{The First Kind of Polynomials}

Consider polynomials of the form $x^{q+1}+x^2+ax^q+bx+c$, whose number of zeros in $\mathbb F_{q^2}$ is denoted by $M$. Then
\begin{equation}\label{eq2}\begin{split}Mq^2&=\sum_{w\in\mathbb F_{q^2}}\sum_{u\in\mathbb F_{q^2}}\chi(u(w^{q+1}+w^2+aw^q+bw+c))\\&=\sum_{u\in\mathbb F_{q^2}}\sum_{w\in\mathbb F_{q^2}}\chi(uw^{q+1}+uw^2+(a^qu^q+bu)w+cu),\end{split}\end{equation}
where the inner sum can be evaluated for every $u\in\mathbb F_{q^2}$ according to Lemma \ref{sum}. In this case we have $D=\Tr(u)^2-4\N(u)=(u^q-u)^2$. Suppose first that $u\ne0$ and $D=0$, which means $u\in\mathbb F_q^*$ and
\[\Tr(u)\theta^q+2u\theta+(a^qu^q+bu)=u(2\theta^q+2\theta+a^q+b)\]
for $\theta\in\mathbb F_{q^2}$. Therefore, if $a^q+b\notin\mathbb F_q$, then the above is nonzero, and
\begin{equation}\label{eq3}\sum_{u\in\mathbb F_q^*}\sum_{w\in\mathbb F_{q^2}}\chi(uw^{q+1}+uw^2+(a^qu^q+bu)w+cu)=0.\end{equation}
If $a^q+b\in\mathbb F_q$, then let $\theta\in\mathbb F_{q^2}$ be such that $2\theta^q+2\theta+a^q+b=0$, and we get
\[\chi(-u\theta^{q+1}-u\theta^2)=\psi(\Tr(-u\theta\Tr(\theta)))=\psi(-u\Tr(\theta)^2)=\psi\left(-\frac u4(a^q+b)^2\right).\]
Furthermore, noticing that $\Tr(u)=2u^{\frac{q+1}2}$ if and only if $u^{\frac{q-1}2}=1$, we have
\[\begin{split}&\mathrel{\phantom{=}}\sum_{u\in\mathbb F_q^*}\sum_{w\in\mathbb F_{q^2}}\chi\left(uw^{q+1}+uw^2+(a^qu^q+bu)w+cu\right)\\&=\sum_{u\in\mathbb F_q^*}\eta(u)qG\psi\left(u\left(-\frac14(a^q+b)^2+\Tr(c)\right)\right)\\&=\eta\left(-\frac14(a^q+b)^2+\Tr(c)\right)qG\sum_{u\in\mathbb F_q^*}\eta(u)\psi(u)\\&=\eta(-(a^q+b)^2+4\Tr(c))qG^2\\&=\eta((a^q+b)^2-4\Tr(c))q^2,\end{split}\]
if $(a^q+b)^2-4\Tr(c)\ne0$, and
\[\sum_{u\in\mathbb F_q^*}\sum_{w\in\mathbb F_{q^2}}\chi(uw^{q+1}+uw^2+(a^qu^q+bu)w+cu)=0\]
otherwise. With $\eta(0)=0$ by convention, we may write
\begin{equation}\label{eq4}\sum_{u\in\mathbb F_q^*}\sum_{w\in\mathbb F_{q^2}}\chi(uw^{q+1}+uw^2+(a^qu^q+bu)w+cu)=\eta((a^q+b)^2-4\Tr(c))q^2.\end{equation}

Suppose now $D\ne0$. Then $D=(u^q-u)^2$ is the discriminant of the minimal polynomial of $u\in\mathbb F_{q^2}\setminus\mathbb F_q$ over $\mathbb F_q$, so $\eta(D)=-1$. It follows from Lemma \ref{sum} that
\begin{equation}\label{eq5}\sum_{w\in\mathbb F_{q^2}}\chi(uw^{q+1}+uw^2+(a^qu^q+bu)w+cu)=q\chi(-u\theta_u^{q+1}-u\theta_u^2+cu),\end{equation}
where
\[\theta_u=\frac{-\Tr(u)(a^qu^q+bu)^q+2u^q(a^qu^q+bu)}{\Tr(u)^2-4\N(u)}.\]
Let $F(u)=-u\theta_u^{q+1}-u\theta_u^2+cu$, which satisfies $F(ju)=jF(u)$ for $j\in\mathbb F_q^*$. Fix an element $\epsilon\in\mathbb F_{q^2}\setminus\mathbb F_q$, so that $u\in\mathbb F_{q^2}\setminus\mathbb F_q$ if and only if $u=j(k+\epsilon)$ for some $j\in\mathbb F_q^*$ and $k\in\mathbb F_q$. Then
\[\sum_{u\in\mathbb F_{q^2}\setminus\mathbb F_q}\chi(F(u))=\sum_{k\in\mathbb F_q}\sum_{j\in\mathbb F_q^*}\chi(F(j(k+\epsilon)))=\sum_{k\in\mathbb F_q}\sum_{j\in\mathbb F_q^*}\psi(j\Tr(F(k+\epsilon))),\]
where
\[\sum_{j\in\mathbb F_q^*}\psi(j\Tr(F(k+\epsilon)))=\begin{cases}q-1&\text{if }\Tr(F(k+\epsilon))=0,\\-1&\text{if }\Tr(F(k+\epsilon))\ne0.\end{cases}\]
Let $M_0$ be the number of $k\in\mathbb F_q$ such that $\Tr(F(k+\epsilon))=0$, so that
\begin{equation}\label{eq6}\sum_{u\in\mathbb F_{q^2}\setminus\mathbb F_q}\chi(F(u))=M_0(q-1)-(q-M_0)=(M_0-1)q.\end{equation}

We proceed to determine the number $M_0$, assuming $a^q+b\notin\mathbb F_q$. Let $\epsilon=a^q-a-b^q+b$. It is easy to see that $\epsilon\notin\mathbb F_q$, $\Tr(\epsilon)=0$ and $\epsilon^2=-\epsilon^{q+1}$. Then for $k\in\mathbb F_q$,
\[\begin{split}\theta_{k+\epsilon}&=\frac{-2k((a^q+b)k+(b-a^q)\epsilon)^q+2(k-\epsilon)((a^q+b)k+(b-a^q)\epsilon)}{4\epsilon^2}\\&=\frac{\epsilon k^2+(-2a^q-a+b^q)\epsilon k+(a^q-b)\epsilon^2}{2\epsilon^2}.\end{split}\]
Observing that
\[\Tr((-2a^q-a+b^q)\epsilon)=\Tr(\Tr(b-a)\epsilon)-\Tr((a^q+b)\epsilon)=-\epsilon^2,\]
we have
\[\begin{split}\Tr(\theta_{k+\epsilon})&=\frac1{2\epsilon^2}\Tr(\epsilon k^2+(-2a^q-a+b^q)\epsilon k+(a^q-b)\epsilon^2)\\&=\frac1{2\epsilon^2}(-\epsilon^2k+\Tr(a-b)\epsilon^2)\\&=-\frac12(k+\Tr(b-a)),\end{split}\]
and
\[\begin{split}&\mathrel{\phantom{=}}\Tr((k+\epsilon)\theta_{k+\epsilon})\\&=\frac1{2\epsilon^2}\Tr((k+\epsilon)(\epsilon k^2+(-2a^q-a+b^q)\epsilon k+(a^q-b)\epsilon^2))\\&=\frac1{2\epsilon^2}(\Tr(\epsilon^2+(-2a^q-a+b^q)\epsilon)k^2-2\Tr(a)\epsilon^2k+\Tr((a^q-b)\epsilon)\epsilon^2)\\&=\frac1{2}(k^2-2\Tr(a)k-\Tr((a+b)\epsilon)).\end{split}\]
Then
\[\begin{split}&\mathrel{\phantom{=}}4\Tr(F(k+\epsilon))\\&=-4\Tr((k+\epsilon)\theta_{k+\epsilon})\Tr(\theta_{k+\epsilon})+4\Tr(c(k+\epsilon))\\&=(k+\Tr(b-a))(k^2-2\Tr(a)k-\Tr((a+b)\epsilon))+4\Tr(c)k+4\Tr(c\epsilon)\\&=k^3-\Tr(3a-b)k^2+(2\Tr(a)\Tr(a-b)-\Tr((a+b)\epsilon))k\\&\mathrel{\phantom{=}}+\Tr(a-b)\Tr((a+b)\epsilon)+4\Tr(c)k+4\Tr(c\epsilon),\end{split}\]
where
\[\begin{split}&\mathrel{\phantom{=}}2\Tr(a)\Tr(a-b)-\Tr((a+b)\epsilon)\\&=2\Tr(a(a-b))+2\Tr(a(a-b)^q)-\Tr((a+b)\epsilon)\\&=4\Tr(a(a-b))-2\Tr(a(a-b))+2\Tr(a(a-b)^q)-\Tr((a+b)\epsilon)\\&=4\Tr(a^2-ab)+2\Tr(a\epsilon)-\Tr((a+b)\epsilon)\\&=4\Tr(a^2-ab)+\Tr((a-b)\epsilon)\\&=4\Tr(a^2-ab)-\epsilon^2,\end{split}\]
and
\[\begin{split}&\mathrel{\phantom{=}}\Tr(a-b)\Tr((a+b)\epsilon)\\&=\Tr((a-b)(a+b)\epsilon)+\Tr((a^q-b^q)(a+b)\epsilon)\\&=\Tr((a^2-b^2+a^qb-ab^q)\epsilon)\\&=4\Tr((a^2-ab)\epsilon)+\Tr((-3a^2-b^2+2a^qb+4ab)\epsilon)\\&=4\Tr((a^2-ab)\epsilon)+\Tr((3a-b)(a^q-a-b^q+b)\epsilon)\\&=4\Tr((a^2-ab)\epsilon)+\Tr(3a-b)\epsilon^2.\end{split}\]
It turns out that
\begin{equation}\label{eq7}4\Tr(F(k+\epsilon))=k^3-\Tr(3a-b)k^2+(4\Tr(\gamma)-\epsilon^2)k+4\Tr(\gamma\epsilon)+\Tr(3a-b)\epsilon^2,\end{equation}
where $\gamma=a^2-ab+c$.

It remains to consider the case $a^q+b\in\mathbb F_q$. Now $a-b=\Tr(a)-(a^q+b)\in\mathbb F_q$ and let $\epsilon$ be another element in $\mathbb F_{q^2}\setminus\mathbb F_q$ with $\Tr(\epsilon)=0$. By the same argument,
\[\begin{split}\theta_{k+\epsilon}&=\frac{-2k((a^q+b)k+(b-a^q)\epsilon)^q+2(k-\epsilon)((a^q+b)k+(b-a^q)\epsilon)}{4\epsilon^2}\\&=\frac{(-2a^q-a+b^q)\epsilon k+(a^q-b)\epsilon^2}{2\epsilon^2}\\&=\frac{(\Tr(b-a)-(a^q+b))\epsilon k+(a^q-b)\epsilon^2}{2\epsilon^2}.\end{split}\]
so that
\[\Tr(\theta_{k+\epsilon})=\frac1{2\epsilon^2}\Tr(a-b)\epsilon^2=a-b,\]
and
\[\begin{split}&\mathrel{\phantom{=}}\Tr((k+\epsilon)\theta_{k+\epsilon})\\&=\frac1{2\epsilon^2}\Tr((k+\epsilon)((\Tr(b-a)-(a^q+b))\epsilon k+(a^q-b)\epsilon^2))\\&=\frac1{2\epsilon^2}(-2\Tr(a)\epsilon^2k+\Tr((a^q-b)\epsilon)\epsilon^2)\\&=-\frac1{2}(2\Tr(a)k+\Tr((a+b)\epsilon)).\end{split}\]
It follows immediately that
\[\begin{split}2\Tr(F(k+\epsilon))&=(a-b)(2\Tr(a)k+\Tr((a+b)\epsilon))+2\Tr(c(k+\epsilon))\\&=2\Tr(a^2-ab+c)k+\Tr((a^2-b^2+2c)\epsilon).\end{split}\]
If $\Tr(a^2-ab+c)=0$, then
\[\begin{split}\Tr((a^2-b^2+2c)\epsilon)&=\Tr((a^2-b^2+2ab-2a^2+2(a^2-ab+c))\epsilon)\\&=-\Tr((a-b)^2\epsilon)+2\Tr((a^2-ab+c)\epsilon)\\&=4(a^2-ab+c)\epsilon.\end{split}\]
Therefore,
\begin{equation}\label{eq8}M_0=\begin{cases}1&\text{if }\Tr(a^2-ab+c)\ne0,\\q&\text{if }a^2-ab+c=0\\0&\text{otherwise.}\end{cases}\end{equation}

Finally, combining \eqref{eq2}, \eqref{eq3}, \eqref{eq4}, \eqref{eq5}, \eqref{eq6} and \eqref{eq7}, we obtain the following.

\begin{proposition}\label{N1}
If $a^q+b\notin\mathbb F_q$, then the number of zeros of $x^{q+1}+x^2+ax^q+bx+c$ in $\mathbb F_{q^2}$ is equal to the number of zeros of
\[x^3-\Tr(3a-b)x^2+(4\Tr(\gamma)-\epsilon^2)x+4\Tr(\gamma\epsilon)+\Tr(3a-b)\epsilon^2\]
in $\mathbb F_q$, where $\gamma=a^2-ab+c$ and $\epsilon=a^q-a-b^q+b$. If $a^q+b\in\mathbb F_q$, then it is
\[\eta((a^q+b)^2-4\Tr(c))+M_0,\]
with $M_0$ as in \eqref{eq8}.
\end{proposition}

\subsection{The Second Kind of Polynomials}

Using the same method, we can discuss the polynomial
\[\alpha x^{2q}+\beta x^2+ax^q+bx+c,\]
with coefficients in $\mathbb F_{q^2}$, where $\N(\alpha)=\N(\beta)\ne0$. Its number of zeros in $\mathbb F_{q^2}$, denoted by $M$, satisfies
\[\begin{split}Mq^2&=\sum_{w\in\mathbb F_{q^2}}\sum_{u\in\mathbb F_{q^2}}\chi(u(\alpha w^{2q}+\beta w^2+aw^q+bw+c))\\&=\sum_{u\in\mathbb F_{q^2}}\sum_{w\in\mathbb F_{q^2}}\chi((\alpha^qu^q+\beta u)w^2+(a^qu^q+bu)w+cu).\end{split}\]
Let $\alpha=-\beta^q\omega^{q-1}$ for some $\omega\in\mathbb F_{q^2}^*$. Then $\alpha x^{2q}+\beta x^2=\omega^{-1}(\omega\beta x^2-(\omega\beta x^2)^q)$. Without loss of generality, let $\alpha=-\beta^q$, so
\[Mq^2=\sum_{u\in\mathbb F_{q^2}}\sum_{w\in\mathbb F_{q^2}}\chi(\beta(u-u^q)w^2+(a^qu^q+bu)w+cu).\]
If $u-u^q=0$, then
\begin{equation}\label{eq14}\sum_{w\in\mathbb F_{q^2}}\chi(\beta(u-u^q)w^2+(a^qu^q+bu)w+cu)=\begin{cases}0&\text{if }(a^q+b)u\ne0,\\q^2\chi(cu)&\text{if }(a^q+b)u=0.\end{cases}\end{equation}
If $u-u^q\ne0$, then $\N(u-u^q)=-(u-u^q)^2$, and by Lemma \ref{sum},
\begin{equation}\label{eq15}\begin{split}&\mathrel{\phantom{=}}\sum_{w\in\mathbb F_{q^2}}\chi(\beta(u-u^q)w^2+(a^qu^q+bu)w+cu)\\&=-\eta(-\N(\beta(u-u^q)))q\chi\left(-\frac{(a^qu^q+bu)^2}{4\beta(u-u^q)}+cu\right)\\&=\eta(\N(\beta))q\chi\left(-\frac{(a^qu^q+bu)^2}{4\beta(u-u^q)}+cu\right).\end{split}\end{equation}
Fix an element $\epsilon\in\mathbb F_{q^2}\setminus\mathbb F_q$ with $\Tr(\epsilon)=0$, and let $M_0$ be the number of $k\in\mathbb F_q$ such that
\[\Tr\left(-\frac{(a^q(k+\epsilon)^q+b(k+\epsilon))^2}{4\beta(k+\epsilon-(k+\epsilon)^q)}+c(k+\epsilon)\right)=0.\]
Then
\[\sum_{u\in\mathbb F_{q^2}\setminus\mathbb F_q}\chi\left(-\frac{(a^qu^q+bu)^2}{4\beta(u-u^q)}+cu\right)=(M_0-1)q,\]
evaluated in the same way as in the last case. Let $A_0=\beta^{-1}(b-a^q)^2-8c$, $A_1=\beta^{-1}(b^2-a^{2q})-4c$ and $A_2=\beta^{-1}(a^q+b)^2$, so that
\[\begin{split}&\mathrel{\phantom{=}}-8\Tr\left(-\frac{(a^q(k+\epsilon)^q+b(k+\epsilon))^2}{4\beta(k+\epsilon-(k+\epsilon)^q)}+c(k+\epsilon)\right)\\&=\Tr\left(\frac{((a^q+b)k+(b-a^q)\epsilon)^2}{\beta\epsilon}-8c(k+\epsilon)\right)\\&=\Tr(\beta^{-1}(a^q+b)^2\epsilon^{-1})k^2+\Tr(2\beta^{-1}(a^q+b)(b-a^q)-8c)k\\&\mathrel{\phantom{=}}+\Tr((\beta^{-1}(b-a^q)^2-8c)\epsilon)\\&=\Tr(A_2\epsilon^{-1})k^2+2\Tr(A_1)k+\Tr(A_0\epsilon)\\&=(A_2-A_2^q)\epsilon^{-1}k^2+2\Tr(A_1)k+(A_0-A_0^q)\epsilon.\end{split}\]
Then $M_0$ is the number of zeros of the polynomial
\[(A_2-A_2^q)\epsilon^{-1}x^2+2\Tr(A_1)x+(A_0-A_0^q)\epsilon\]
in $\mathbb F_q$. In particular, if $A_2\notin\mathbb F_q$, then it is a quadratic polynomial over $\mathbb F_q$, whose discriminant is
\[4\Tr(A_1)^2-4(A_2-A_2^q)(A_0-A_0^q)=4\Tr(A_1)^2+4\Tr(A_0(A_2^q-A_2)).\]
If $a^q+b=0$, then
\[(A_2-A_2^q)\epsilon^{-1}x^2+2\Tr(A_1)x+(A_0-A_0^q)\epsilon=-8\Tr(c)x+4(\beta^{-1}a^{2q}-\beta^{-q}a^2+2c^q-2c)\epsilon.\]
Hence, using \eqref{eq14} and \eqref{eq15} we have
\[\sum_{u\in\mathbb F_q}\sum_{w\in\mathbb F_{q^2}}\chi(\beta(u-u^q)w^2+(a^qu^q+bu)w+cu)=\begin{cases}q^2&\text{if }a^q+b\ne0,\\q^3&\text{if }a^q+b=\Tr(c)=0,\\0&\text{otherwise},\end{cases}\]
with
\[\sum_{u\in\mathbb F_{q^2}\setminus\mathbb F_q}\sum_{w\in\mathbb F_{q^2}}\chi(\beta(u-u^q)w^2+(a^qu^q+bu)w+cu)=\eta(\N(\beta))(M_0-1)q^2,\]
and obtain the following result.

\begin{proposition}\label{N2}
Suppose $\beta\ne0$. Let $A_0=\beta^{-1}(b-a^q)^2-8c$, $A_1=\beta^{-1}(b^2-a^{2q})-4c$ and $A_2=\beta^{-1}(a^q+b)^2$. If $a^q+b\ne0$, then the number of zeros of $-\beta^qx^{2q}+\beta x^2+ax^q+bx+c$ in $\mathbb F_{q^2}$ is equal to
\begin{itemize}
\item $1+\eta(\N(\beta))\eta(\Tr(A_1)^2+\Tr(A_0(A_2^q-A_2)))$ if $A_2\notin\mathbb F_q$;
\item $1$ if $A_2\in\mathbb F_q$ and $\Tr(A_1)\ne0$;
\item $1-\eta(\N(\beta))$ if $A_2\in\mathbb F_q$, $\Tr(A_1)=0$ and $A_0\notin\mathbb F_q$;
\item $1+\eta(\N(\beta))(q-1)$ if $A_2\in\mathbb F_q$, $\Tr(A_1)=0$ and $A_0\in\mathbb F_q$.
\end{itemize}
If $a^q+b=0$, then that is
\begin{itemize}
\item $0$ if $\Tr(c)\ne0$;
\item $q-\eta(\N(\beta))$ if $\Tr(c)=0$ and $\beta^{-1}a^{2q}-\beta^{-q}a^2-4c\ne0$;
\item $q+\eta(\N(\beta))(q-1)$ if $\Tr(c)=\beta^{-1}a^{2q}-\beta^{-q}a^2-4c=0$.
\end{itemize}
\end{proposition}

\section{Permutation Rational Functions over $\mathbb F_{q^2}$}

In this section, we consider
\[\frac{x^{q+1}+\alpha x^{2q}+\beta x^2+ax^q+bx+c}{x^q+x+d}\]
or
\[\frac{\alpha x^{2q}+\beta x^2+ax^q+bx+c}{x^q+x+d}\]
with coefficients in $\mathbb F_{q^2}$ and $d\notin\mathbb F_q$. The necessary and sufficient conditions for permutation rational functions of $\mathbb F_{q^2}$ in some specific forms will be given after the following lemma.

\begin{lemma}[\cite{williams1975note}]\label{cubic}
Let $g(x)=x^3+Ax^2+Bx+C$ be a cubic polynomial over $\mathbb F_q$, whose discriminant is
\[D=A^2B^2-4B^3-4A^3C-27C^2+18ABC.\]
If $q$ is odd and $D\ne0$, then $g$ has exactly one zero in $\mathbb F_q$ if and only if $\eta(D)=-1$. If $q$ is even, then $g$ has exactly one zero in $\mathbb F_q$ if and only if either $A^2+B=AB+C=0$, or $AB+C\ne0$ with $x^2+x+\frac{(A^2+B)^3}{(AB+C)^2}+1$ irreducible over $\mathbb F_q$.
\end{lemma}

\begin{theorem}
Let $q$ be odd, $\gamma=a^2-ab+c$, $\delta=a-b+d$ and $\epsilon=a^q-a-b^q+b$. Then 
\[\frac{x^{q+1}+x^2+ax^q+bx+c}{x^q+x+d}\]
is a permutation rational function of $\mathbb F_{q^2}$ if and only if $\gamma\in\{0,-\epsilon^2\}$, $\delta=0$ and $\epsilon\ne0$.
\end{theorem}
\begin{proof}
Consider the number of zeros of
\[x^{q+1}+x^2+ax^q+bx+c+t(x^q+x+d)=x^{q+1}+x^2+(a+t)x^q+(b+t)x+c+dt\]
in $\mathbb F_{q^2}$ for $t\in\mathbb F_{q^2}$. Note that $a^q+b\in\mathbb F_q$ if and only if $(a+t)^q+b+t\in\mathbb F_q$, for every $t\in\mathbb F_{q^2}$.

First, assume $a^q+b\notin\mathbb F_q$. By Proposition \ref{N1}, the rational function permutes $\mathbb F_{q^2}$ if and only if
\begin{equation}\label{eq9}x^3-\Tr(3a-b+2t)x^2+(4\Tr(\gamma+\delta t)-\epsilon^2)x+4\Tr((\gamma+\delta t)\epsilon)+\Tr(3a-b+2t)\epsilon^2\end{equation}
has exactly one zero in $\mathbb F_q$ for every $t\in\mathbb F_{q^2}$. The discriminant of the above cubic polynomial is
\[D_t=A_t^2B_t^2-4B_t^3-4A_t^3C_t-27C_t^2+18A_tB_tC_t,\]
where $A_t=-\Tr(3a-b+2t)$, $B_t=4\Tr(\gamma+\delta t)-\epsilon^2$ and $C_t=4\Tr((\gamma+\delta t)\epsilon)+\Tr(3a-b+2t)\epsilon^2$. Provided that $D_t\ne0$, the polynomial \eqref{eq9} has one zero in $\mathbb F_q$ if and only if $\eta(D_t)=-1$, as a result of Lemma \ref{cubic}. Therefore, the rational function permutes $\mathbb F_{q^2}$ only if $\eta(D_t)=-1$ for every $t\in\mathbb F_{q^2}$ such that $D_t\ne0$.

Let $\delta\ne0$ and $q>9$. If $\delta\notin\mathbb F_q$, then the map $t\in\mathbb F_{q^2}\mapsto(t_0,t_1)\in\mathbb F_q\times\mathbb F_q$, where $t_0=\Tr(t)$ and $t_1=\Tr(\delta t)$, is surjective. In this case, fix some $t_0\in\mathbb F_q$ and view $A_t$, $B_t$ and $C_t$ as polynomials in $t_1$ over $\mathbb F_q$, in which sense $\deg(A_t)=0$, $\deg(B_t)=1$ and $\deg(C_t)\le1$. Thus $\deg(D_t)=3$ and the number of $t_1\in\mathbb F_q$ such that $D_t=0$ is at most $3$. If for every $t_1\in\mathbb F_q$ the polynomial \eqref{eq9} has exactly one zero in $\mathbb F_q$, then
\[\left|\sum_{t_1\in\mathbb F_q}\eta(D_t)\right|\ge q-3,\]
while by the well-known Weil bound for character sums (\cite[Theorem 5.41]{lidl1997}),
\[\left|\sum_{t_1\in\mathbb F_q}\eta(D_t)\right|\le 2\sqrt q< q-3.\]
Suppose $\delta\in\mathbb F_q$. Similarly, the map $t\in\mathbb F_{q^2}\mapsto(t_0,t_1)\in\mathbb F_q\times\mathbb F_q$, where $t_0=\Tr(t)$ and $t_1=\Tr(\epsilon t)$, is surjective. Fix arbitrary $t_0\in\mathbb F_q$ and we have
\[A_t=-\Tr(3a-b)-2t_0,\]
\[B_t=4\Tr(\gamma)-\epsilon^2+4\delta t_0,\]
\[C_t=4\Tr(\gamma\epsilon)+4\delta t_1+\Tr(3a-b)\epsilon^2+2\epsilon^2t_0.\]
Now $D_t$ is a quadratic polynomial in $t_1$ over $\mathbb F_q$ with $\deg(A_t)=\deg(B_t)=0$ and $\deg(C_t)=1$, whose discriminant (via the map $C_t\mapsto t_1$) is
\[(-4A_t^3+18A_tB_t)^2+108(A_t^2B_t^2-4B_t^3).\]
This, written as a polynomial in $t_0$ over $\mathbb F_q$, has degree $6$. Since $q>6$, let $t_0$ be such that the discriminant is nonzero, and thus
\[\left|\sum_{t_1\in\mathbb F_q}\eta(D_t)\right|\le\sqrt q< q-2.\]
This implies that $\eta(D_t)=1$ for some $t\in\mathbb F_{q^2}$ and there is no such permutation rational function.

Let $\delta=0$ and $q>16$. Then $B_t=4\Tr(\gamma)-\epsilon^2=\epsilon^2k_0$ for some $k_0\in\mathbb F_q$, $C_t=4\Tr(\gamma\epsilon)-\epsilon^2A_t$ and
\[\begin{split}D_t&=4\epsilon^2A_t^4-16\Tr(\gamma\epsilon)A_t^3+(B_t^2-18\epsilon^2B_t-27\epsilon^4)A_t^2\\&\mathrel{\phantom{=}}+72\Tr(\gamma\epsilon)(B_t+3\epsilon^2)A_t-4B_t^3-432\Tr(\gamma\epsilon)^2\\&=4\epsilon^2A_t^4-16\Tr(\gamma\epsilon)A_t^3+\epsilon^4(k_0^2-18k_0-27)A_t^2\\&\mathrel{\phantom{=}}+72\epsilon^2\Tr(\gamma\epsilon)(k_0+3)A_t-4\epsilon^6k_0^3-432\Tr(\gamma\epsilon)^2.\end{split}\]
Note that $A_t$ runs through $\mathbb F_q$ as $t$ runs through $\mathbb F_{q^2}$, so consider the polynomial over $\mathbb F_q$:
\[\begin{split}D^\prime(x)&=4\epsilon^2x^4-16\Tr(\gamma\epsilon)x^3+\epsilon^4(k_0^2-18k_0-27)x^2\\&\mathrel{\phantom{=}}+72\epsilon^2\Tr(\gamma\epsilon)(k_0+3)x-4\epsilon^6k_0^3-432\Tr(\gamma\epsilon)^2.\end{split}\]
If it is not a square of any quadratic polynomial over $\mathbb F_q$ multiplied by a constant, then again by the Weil bound,
\[\left|\sum_{t\in\mathbb F_q}\eta(D^\prime(t))\right|\le 3\sqrt q< q-4.\]
Hence, to induce a permutation rational function, one must have $D^\prime(x)=\epsilon^2(2x^2+\tau_1x+\tau_0)^2$ for some $\tau_1,\tau_0\in\mathbb F_q$. Comparing the coefficients, we get $\epsilon^2\tau_1=-4\Tr(\gamma\epsilon)$, as well as
\begin{equation}\label{eq10}\epsilon^2(\tau_1^2+4\tau_0)=\epsilon^4(k_0^2-18k_0-27),\end{equation}
\begin{equation}\label{eq11}2\epsilon^2\tau_1\tau_0=72\epsilon^2\Tr(\gamma\epsilon)(k_0+3),\end{equation}
and
\begin{equation}\label{eq12}\epsilon^2\tau_0^2=-4\epsilon^6k_0^3-432\Tr(\gamma\epsilon)^2.\end{equation}
Suppose $\tau_1\ne0$. Then $\tau_0=-9\epsilon^2(k_0+3)$ by \eqref{eq11}, and with \eqref{eq10},
\[\tau_1^2=\epsilon^2(k_0^2-18k_0-27)-4\tau_0=\epsilon^2(k_0^2+18k_0+81).\]
Moreover, we have
\[27\epsilon^4\tau_1^2=432\Tr(\gamma\epsilon)^2=-4\epsilon^6k_0^3-\epsilon^2\tau_0^2=-\epsilon^6(4k_0^3+81(k_0+3)^2)\]
by \eqref{eq12}, and thus,
\[4k_0^3+81(k_0+3)^2+27(k_0^2+18k_0+81)=-27\epsilon^{-2}\tau_1^2+27(k_0^2+18k_0+81)=0.\]
This means
\[4(k_0+9)^3=4k_0^3+81(k_0+3)^2+27(k_0^2+18k_0+81)=0\]
and $\tau_1^2=\epsilon^2(k_0^2+18k_0+81)=0$, a contradiction. Accordingly, let $\tau_1=0$, so that $\Tr(\gamma\epsilon)=0$, $4\tau_0=\epsilon^2(k_0^2-18k_0-27)$ and $\tau_0^2=-4\epsilon^4k_0^3$ by \eqref{eq10} and \eqref{eq12}. A simple calculation yields
\[(k_0+1)(k_0+9)^3=(k_0^2-18k_0-27)^2+64k_0^3=0.\]
Recall that $4\Tr(\gamma)-\epsilon^2=\epsilon^2k_0$, so $\Tr(\gamma)\in\{0,-2\epsilon^2\}$, while $\Tr(\gamma\epsilon)=0$. Since $\epsilon\not\mathbb F_q$, it is easy to see that $\gamma\in\{0,-\epsilon^2\}$. Conversely, if $\gamma=0$, then
\[D_t=4\epsilon^2(A_t^4-2\epsilon^2A_t^2+\epsilon^4)=4\epsilon^2(A_t^2-\epsilon^2)^2,\]
and $\eta(D_t)=-1$ for every $t\in\mathbb F_{q^2}$ since $\eta(\epsilon^2)=-1$ and $A_t\in\mathbb F_q$. If $\gamma=-\epsilon^2$, then
\[D_t=4\epsilon^2(A_t^4+54\epsilon^2A_t^2+9^3\epsilon^6)=4\epsilon^2(A_t^2+27\epsilon^2)^2,\]
and when $D_t=0$ the polynomial \eqref{eq9} becomes
\[x^3+A_tx^2-9\epsilon^2x-\epsilon^2A_t=x^3+A_tx^2+\frac{A_t^2}3x+\frac{A_t^3}{27}=\left(x+\frac{A_t}3\right)^3\]
if the characteristic of $\mathbb F_q$ is not $3$, and becomes $x^3$ otherwise. In either case, the polynomial \eqref{eq9} has only one zero in $\mathbb F_q$ for every $t\in\mathbb F_{q^2}$, and subsequently, this induces a permutation rational function of $\mathbb F_{q^2}$. The rest ($\delta\ne0$, $q\le9$ and $\delta=0$, $q<16$) can be verified by exhaustive search.

Assume $a^q+b\in\mathbb F_q$. It follows from Proposition \ref{N1} that the number of zeros of
\[x^{q+1}+x^2+(a+t)x^q+(b+t)x+c+dt\]
in $\mathbb F_{q^2}$ is
\[M=\eta((a^q+b)^2-4\Tr(c)+\Tr(t)^2+2(a^q+b)\Tr(t)-4\Tr(dt))+M_0,\]
where
\[M_0=\begin{cases}1&\text{if }\Tr(\gamma+\delta t)\ne0,\\q&\text{if }\gamma+\delta t=0,\\0&\text{otherwise.}\end{cases}\]
If $\delta\ne0$, then there exists $t\in\mathbb F_{q^2}$ such that $\gamma+\delta t=0$, in which case $M>1$. If $\delta=0$, then $M_0$ is independent of $t$, while $\eta((a^q+b)^2-4\Tr(c)+\Tr(t)^2+2(a^q+b)\Tr(t)-4\Tr(dt))$ is not, as easily seen. In either case $M\ne1$ for some $t\in\mathbb F_{q^2}$. This completes the proof.
\end{proof}

Consider
\[\frac{x^{q+1}+\alpha x^{2q}+\beta x^2+ax^q+bx+c}{x^q+x+d}\]
with $\alpha+\beta=1$ and $\N(\alpha)\ne\N(\beta)$. Let $L(x)=\alpha x^q-\beta^qx$, so that
\[\begin{split}&\mathrel{\phantom{=}}L(x)^{q+1}+\alpha L(x)^{2q}+\beta L(x)^2\\&=(\N(\alpha)-2\alpha\beta+\N(\beta))x^{q+1}+\alpha\beta(\alpha+\beta-1)x^{2q}+(\alpha^{2q+1}+\beta^{2q+1}-\alpha^q\beta^q)x^2.\end{split}\]
Then
\[\begin{split}&\mathrel{\phantom{=}}\frac{L(x)^{q+1}+\alpha L(x)^{2q}+\beta L(x)^2+aL(x)^q+bL(x)+c}{L(x)^q+L(x)+d}\\&=\frac{(\N(\alpha)-2\alpha\beta+\N(\beta))x^{q+1}+(\alpha^{2q+1}+\beta^{2q+1}-\alpha^q\beta^q)x^2+aL(x)^q+bL(x)+c}{(\alpha-\beta)x^q+(\alpha^q-\beta^q)x+d}.\end{split}\]
Substitute $\frac x{\alpha^q-\beta^q}$ for $x$ to get
\[\frac{\frac{\N(\alpha)-2\alpha\beta+\N(\beta)}{(\alpha^q-\beta^q)^{q+1}}x^{q+1}+\frac{\alpha^{2q+1}+\beta^{2q+1}-\alpha^q\beta^q}{(\alpha^q-\beta^q)^2}x^2+aL\left(\frac x{\alpha^q-\beta^q}\right)^q+bL\left(\frac x{\alpha^q-\beta^q}\right)+c}{x^q+x+d},\]
where
\[\frac{\N(\alpha)-2\alpha\beta+\N(\beta)}{(\alpha^q-\beta^q)^{q+1}}=\frac{\alpha^{2q+1}+\beta^{2q+1}-\alpha^q\beta^q}{(\alpha^q-\beta^q)^2},\]
since
\[\begin{split}&\mathrel{\phantom{=}}(\alpha^{2q+1}+\beta^{2q+1}-\alpha^q\beta^q)(\alpha-\beta)-(\N(\alpha)-2\alpha\beta+\N(\beta))(\alpha^q-\beta^q)\\&=\alpha^{2q+2}+\alpha\beta^{2q+1}-\alpha^{2q+1}\beta-\beta^{2q+2}-\N(\alpha)\beta^q+\alpha^q\N(\beta)\\&\mathrel{\phantom{=}}-(\alpha^{2q+1}-2\N(\alpha)\beta+2\alpha\N(\beta)-\beta^{2q+1}-\N(\alpha)\beta^q+\alpha^q\N(\beta))\\&=\alpha^{2q+2}+\alpha\beta^{2q+1}-\alpha^{2q+1}\beta-\beta^{2q+2}-\alpha^{2q+1}+2\N(\alpha)\beta-2\alpha\N(\beta)+\beta^{2q+1}\\&=(\alpha+\beta-1)(\alpha^{2q+1}-\beta^{2q+1})+2\alpha\beta^{2q+1}-2\alpha^{2q+1}\beta+2\N(\alpha)\beta-2\alpha\N(\beta)\\&=(\alpha+\beta-1)(\alpha^{2q+1}-\beta^{2q+1})+2\alpha\beta(\alpha+\beta-1)^q(\beta-\alpha)^q\\&=0.\end{split}\]
Therefore, the above results can be applied in this case for permutation rational functions.

Similarly, Proposition \ref{N1} can be used for another form of rational functions as follows.

\begin{proposition}
Let $q$ be odd, $\beta\ne1$ and $\N(\beta)=1$. Then
\[\frac{x^{q+1}+\beta x^2+ax^q+bx+c}{x^q+x+d}\]
is not a permutation rational function of $\mathbb F_{q^2}$.
\end{proposition}
\begin{proof}
Let $\beta=\omega^{1-q}$ for some $\omega\in\mathbb F_{q^2}$. For $t\in\mathbb F_{q^2}$, substituting $\omega^{-1}x$ for $x$ in
\[x^{q+1}+\beta x^2+ax^q+bx+c+t(x^q+x+d)\]
and multiplying it by $\omega^{q+1}$, we get
\[x^{q+1}+x^2+\omega(a+t)x^q+\omega^q(b+t)x+\omega^{q+1}(c+dt).\]
Since $\beta\ne1$, there exists $t\in\mathbb F_{q^2}$ such that
\[(\omega^q-\omega)\Tr(t)=\omega(a+b^q)-\omega^q(a^q+b),\]
which means
\[\omega^q(a^q+b+\Tr(t))=\omega(a+b^q+\Tr(t)).\]
For all such $t$, clearly $\Tr(\omega^{q+1}(c+dt))$ takes all values in $\mathbb F_q$ as $\omega^{q+1}d\notin\mathbb F_q$. Therefore,  we have
\[\omega^q(a+t)^q+\omega^q(b+t)=\omega^q(a^q+b+\Tr(t))\in\mathbb F_q,\]
and
\[\eta((\omega^q(a+t)^q+\omega^q(b+t))^2-4\Tr(\omega^{q+1}(c+dt)))=-1\]
for some $t\in\mathbb F_{q^2}$, and it follows from Proposition \ref{N1} that the number of zeros of
\[x^{q+1}+x^2+\omega(a+t)x^q+\omega^q(b+t)x+\omega^{q+1}(c+dt)\]
in $\mathbb F_{q^2}$ is not one.
\end{proof}

In the last section, the number of zeros of $x^{q+1}+x^2+ax^q+bx+c$ in $\mathbb F_{q^2}$ is determined using character sums, only in the case of odd $q$. In fact, there is another method for any characteristic, which, however, can not be applied to polynomials of more general forms.

\begin{lemma}
If $a-b\notin\mathbb F_q$, then the number of zeros of $x^{q+1}+x^2+ax^q+bx+c$ in $\mathbb F_{q^2}$ is equal to the number of zeros of
\[x^3+\Tr(b)x^2+(\Tr(c)-\N(a)+\N(b))x+\Tr((b-a)^q(a^2-ab+c))+\Tr(a)\N(b-a)\]
in $\mathbb F_q$. If $a-b\in\mathbb F_q$ and $\Tr(a^2-ab+c)\ne0$, then it is the number of zeros of
\[x^2+(a^q+b)x+\Tr(c)\]
in $\mathbb F_q$. If $a-b\in\mathbb F_q$ and $\Tr(a^2-ab+c)=0$, then it is
\[\begin{cases}1&\text{if }a^2-ab+c\ne0\text{ and }a^q+2a-b\ne0,\\0&\text{if }a^2-ab+c\ne0\text{ and }a^q+2a-b=0,\\q+1&\text{if }a^2-ab+c=0\text{ and }a^q+2a-b\ne0,\\q&\text{if }a^2-ab+c=0\text{ and }a^q+2a-b=0.\end{cases}\]
\end{lemma}
\begin{proof}
Note that
\[x^{q+1}+x^2+ax^q+bx+c=(x+a)(x^q+x-a+b)+a^2-ab+c.\]
If $a-b\in\mathbb F_q$ and $a^2-ab+c=0$, then clearly its number of zeros in $\mathbb F_{q^2}$ is
\[\begin{cases}q+1&\text{if }a^q+2a-b\ne0,\\q&\text{if }a^q+2a-b=0.\end{cases}\]
Suppose that $a-b\notin\mathbb F_q$ or $a^2-ab+c\ne0$ from now on. If $x_0\in\mathbb F_{q^2}$ is a zero of the equation
\begin{equation}\label{eq0}(x+a)(x^q+x-a+b)+a^2-ab+c=0,\end{equation}
then $x_0^q+x_0-a+b\ne0$, and for $\tau=\Tr(x_0)$,
\[\tau+\Tr(a)+\Tr\left(\frac{a^2-ab+c}{\tau-a+b}\right)=\Tr\left(x_0+a+\frac{a^2-ab+c}{x_0^q+x_0-a+b}\right)=0.\]
Conversely, if there exists $\tau\in\mathbb F_q$ such that
\begin{equation}\label{eq1}\tau+\Tr(a)+\Tr\left(\frac{a^2-ab+c}{\tau-a+b}\right)=0\qquad(\tau-a+b\ne0),\end{equation}
then let
\[x_0=-a-\frac{a^2-ab+c}{\tau-a+b},\]
so that $\Tr(x_0)=\tau$ and $x_0$ is a zero of \eqref{eq0}. Moreover, if $x_0$ and $x_1$ are both zeros of \eqref{eq0} in $\mathbb F_{q^2}$ with $\Tr(x_1)=\Tr(x_0)$, then
\[x_0+a+\frac{a^2-ab+c}{\Tr(x_0)-a+b}=x_1+a+\frac{a^2-ab+c}{\Tr(x_1)-a+b}\]
and obviously $x_0=x_1$. Hence there is a one-to-one correspondence between the zeros of \eqref{eq0} in $\mathbb F_{q^2}$ and those $\tau\in\mathbb F_q$ satisfying \eqref{eq1}. If $a-b\notin\mathbb F_q$, then multiplying both sides of \eqref{eq1} by $\N(\tau-a+b)$ we obtain its equivalent form
\[(\tau+\Tr(a))(\tau^2+\Tr(b-a)\tau+\N(b-a))+\Tr((a^2-ab+c)(\tau-a^q+b^q))=0;\]
that is,
\[\tau^3+\Tr(b)\tau^2+(\N(b)-\N(a)+\Tr(c))\tau+\Tr(a)\N(b-a)+\Tr((a^2-ab+c)(b^q-a^q)).\]
If $a-b\in\mathbb F_q$ and $\Tr(a^2-ab+c)\ne0$, then \eqref{eq1} is equivalent to
\[(\tau+\Tr(a))(\tau-a+b)+\Tr(a^2-ab+c)=0,\]
where
\[(\tau+\Tr(a))(\tau-a+b)+\Tr(a^2-ab+c)=\tau^2+(a^q+b)\tau+\Tr(c).\]
If $a-b\in\mathbb F_q$ and $\Tr(a^2-ab+c)=0$, then the result follows immediately from \eqref{eq1}.
\end{proof}

Having obtained the number of zeros of $x^{q+1}+x^2+ax^q+bx+c$ in $\mathbb F_{q^2}$ in the case of even characteristic, we can now give a sufficient condition for
\[\frac{x^{q+1}+\alpha x^{2q}+\beta x^2+ax^q+bx+c}{x^q+x+d}\]
being a permutation rational function of $\mathbb F_{q^2}$, where $\alpha+\beta=1$ and $\N(\alpha)\ne\N(\beta)$. For simplicity, let $\alpha=0$ and $\beta=1$, in which case the condition coincides with that of odd characteristic. 

\begin{theorem}
Let $q$ be even, $a^2+ab+c\in\{0,\Tr(a+b)^2\}$ and $a+b+d=0$. Then
\[\frac{x^{q+1}+x^2+ax^q+bx+c}{x^q+x+d}\]
is a permutation rational function of $\mathbb F_{q^2}$.
\end{theorem}
\begin{proof}
The case $a^2+ab+c=a+b+d=0$ is trivial, since then
\[\frac{(x+a)(x^q+x+a+b)+a^2+ab+c}{x^q+x+d}=x+a.\]
We claim that
\[x^3+\Tr(\beta)x^2+(\Tr(\alpha^2+\alpha\beta)+\N(\alpha)+\N(\beta))x+\Tr(\alpha+\beta)^3+\Tr(\alpha)\N(\alpha+\beta)\]
has exactly one zero in $\mathbb F_q$ for all $\alpha,\beta\in\mathbb F_{q^2}$ such that $\Tr(\alpha+\beta)\ne0$. Assume $a^2+ab+c=\Tr(a+b)^2$ and $a+b+d=0$. Let $\alpha=a+t$ and $\beta=b+t$ for $t\in\mathbb F_{q^2}$. Then $\Tr(\alpha+\beta)=\Tr(a+b)=\Tr(d)\ne0$, and
\[\begin{split}&\mathrel{\phantom{=}}\Tr(\alpha^2+\alpha\beta)+\N(\alpha)+\N(\beta)\\&=\Tr(a^2+ab+(a+b)t)+\N(a+t)+\N(b+t)\\&=\Tr(c+dt)+\N(a+t)+\N(b+t),\end{split}\]
with
\[\begin{split}&\mathrel{\phantom{=}}\Tr(\alpha+\beta)^3+\Tr(\alpha)\N(\alpha+\beta)\\&=\Tr(a+b)^3+\Tr(a+t)\N(a+b)\\&=\Tr((a+b)^q((a+t)^2+(a+t)(b+t)+c+dt))+\Tr(a+t)\N(a+b).\end{split}\]
The claim, combined with the preceding proposition, completes the proof.

For the claim, let $A=\Tr(\beta)$, $B=\Tr(\alpha^2+\alpha\beta)+\N(\alpha)+\N(\beta)$ and $C=\Tr(\alpha+\beta)^3+\Tr(\alpha)\N(\alpha+\beta)$, which satisfy
\[\begin{split}&\mathrel{\phantom{=}}\Tr(\alpha)(A^2+B)+\Tr(\beta)^3\\&=\Tr(\alpha)(\Tr(\alpha^2+\alpha\beta+\beta^2)+\N(\alpha)+\N(\beta))+\Tr(\beta)^3\\&=\Tr(\alpha)\Tr(\alpha+\beta)^2+\Tr(\beta)^3+\Tr(\alpha)(\Tr(\alpha\beta)+\N(\alpha)+\N(\beta))\\&=\Tr(\alpha+\beta)^3+\Tr(\alpha)^2\Tr(\beta)+\Tr(\alpha)(\Tr(\alpha)\Tr(\beta)+\N(\alpha+\beta))\\&=\Tr(\alpha+\beta)^3+\Tr(\alpha)\N(\alpha+\beta)\\&=C,\end{split}\]
and then
\[\Tr(\alpha+\beta)(A^2+B)=\Tr(\alpha)(A^2+B)+\Tr(\beta)^3+\Tr(\beta)B=AB+C.\]
Since $\Tr(\alpha+\beta)\ne0$, we have $A^2+B=0$ if $AB+C=0$. By Lemma \ref{cubic}, it remains to prove that
\[x^2+x+\frac{(A^2+B)^3}{(AB+C)^2}+1,\]
is irreducible over $\mathbb F_q$ provided $AB+C\ne0$. Assume, on the contrary, that there exists some $x_0\in\mathbb F_q$ such that
\[x_0^2+x_0+\frac{(A^2+B)^3}{(AB+C)^2}+1=0,\]
where
\[\frac{(A^2+B)^3}{(AB+C)^2}+1=\frac{A^2+B+\Tr(\alpha+\beta)^2}{\Tr(\alpha+\beta)^2}=\frac{\Tr(\alpha\beta)+\N(\alpha)+\N(\beta)}{\Tr(\alpha+\beta)^2}.\]
Let $x_1=\Tr(\alpha+\beta)x_0+\Tr(\alpha)$, so that
\[\begin{split}&\mathrel{\phantom{=}}x_1^2+\Tr(\alpha+\beta)x_1+\N(\alpha+\beta)\\&=\Tr(\alpha+\beta)^2x_0^2+\Tr(\alpha)^2+\Tr(\alpha+\beta)^2x_0+\Tr(\alpha+\beta)\Tr(\alpha)+\N(\alpha+\beta)\\&=\Tr(\alpha+\beta)^2(x_0^2+x_0)+\Tr(\alpha)\Tr(\beta)+\N(\alpha+\beta)\\&=\Tr(\alpha\beta)+\N(\alpha)+\N(\beta)+\Tr(\alpha)\Tr(\beta)+\N(\alpha+\beta)\\&=0,\end{split}\]
which means $x_1=a+b$ or $x_1=a^q+b^q$, but $x_1\in\mathbb F_q$, a contradiction.
\end{proof}

Finally, rational functions derived from the second kind of $q$-quadratic polynomials are studied. However, it turns out that only a finite number of such permutation rational functions exist.

\begin{proposition}
Let $\alpha=-\beta^q\ne0$. Then
\[\frac{\alpha x^{2q}+\beta x^2+ax^q+bx+c}{x^q+x+d}\]
is a permutation rational function of $\mathbb F_{q^2}$ if and only if $q=3$, $a^q+b\notin\mathbb F_q$, $\beta^2=(b^2-a^2-\beta c)^2=1$ and $\beta(a-b)-d=0$.
\end{proposition}
\begin{proof}
If $q$ is even, then it is obvious. Let $q$ be odd. If $a^q+b\in\mathbb F_q$, then $a^q+b+\Tr(t)=0$ and $\Tr(c+dt)\ne0$ for some $t\in\mathbb F_{q^2}$, so that
\begin{equation}\label{eq13}\alpha x^{2q}+\beta x^2+ax^q+bx+c+t(x^q+x+d)\end{equation}
has no zero in $\mathbb F_{q^2}$, as a consequence of Proposition \ref{N2}. Suppose $q>5$, $a^q+b\notin\mathbb F_q$ and \eqref{eq13} has exactly one zero in $\mathbb F_{q^2}$ for every $t\in\mathbb F_{q^2}$. Then we have, by Proposition \ref{N2}, either
\[A_2\notin\mathbb F_q\quad\text{and}\quad\Tr(A_1)^2+\Tr(A_0(A_2^q-A_2))=0,\]
or
\[A_2\in\mathbb F_q\quad\text{and}\quad\Tr(A_1)\ne0,\]
where $A_0=\beta^{-1}(b-a^q+t-t^q)^2-8(c+dt)$, $A_1=\beta^{-1}(b^2-a^{2q}+2bt+t^2-2a^qt^q-t^{2q})-4(c+dt)$ and $A_2=\beta^{-1}(a^q+b+\Tr(t))^2$. Note that
\[\begin{split}A_2-A_2^q&=(\beta^{-1}-\beta^{-q})\Tr(t)^2+2(\beta^{-1}(a^q+b)-\beta^{-q}(a+b^q))\Tr(t)\\&\mathrel{\phantom{=}}+\beta^{-1}(a^q+b)^2-\beta^{-q}(a+b^q)^2,\end{split}\]
and thus there are at least $q^2-2q$ number of $t\in\mathbb F_{q^2}$ such that $A_2\notin\mathbb F_q$. Regard $t$ as an indeterminate over $\mathbb F_{q^2}$ by abuse of notation. Then the polynomial
\[\Tr(A_1)^2+\Tr(A_0(A_2^q-A_2))\]
in $t$ has degree at most $4q$ and at least $q^2-2q$ number of zeros in $\mathbb F_{q^2}$. This implies that it is a zero polynomial, for otherwise $q^2-2q\le4q$. Moreover, we have
\[A_0\equiv\beta^{-1}(b-a^q)^2-8(c+dt)\pmod{t^q-t},\]
\[A_1\equiv\beta^{-1}(b^2-a^{2q}+2(b-a^q)t)-4(c+dt)\pmod{t^q-t},\]
and
\[A_2\equiv\beta^{-1}(a^q+b+2t)^2\pmod{t^q-t}.\]
Look at the cubic coefficient of $\Tr(A_1)^2+\Tr(A_0(A_2^q-A_2))$ modulo $t^q-t$, which is
\[-32\Tr(d(\beta^{-q}-\beta^{-1}))=-32(d-d^q)(\beta^{-q}-\beta^{-1}).\]
Then $\beta\in\mathbb F_q$ as $d\notin\mathbb F_q$. Now
\[A_2-A_2^q=2\beta^{-1}(a^q+b-a-b^q)\Tr(t)+\beta^{-1}((a^q+b)^2-(a+b^q)^2)\]
and $a^q+b\notin\mathbb F_q$, so $A_2\in\mathbb F_q$ with $\Tr(A_1)\ne0$ for some $t\in\mathbb F_{q^2}$, but
\[\Tr(A_1)^2+\Tr(A_0(A_2^q-A_2))=0,\]
a contradiction.

The case $q\le5$ can be verified by exhaustive search.
\end{proof}

\bibliographystyle{abbrv}
\bibliography{references}

\end{document}